\documentclass[review,12pt]{elsarticle}
\biboptions{numbers,sort&compress}
\usepackage{amssymb}
\usepackage{graphicx}
\usepackage{amsmath}
\usepackage{lineno,hyperref}
\usepackage{amsmath}
\usepackage{appendix}
\usepackage{graphicx}
\usepackage{setspace}
\usepackage{geometry}
\usepackage{setspace}

\renewcommand{\theequation}{\arabic{section}.\arabic{equation}}% 公式编号中带章节号
\numberwithin{equation}{section}
\newtheorem{theorem}{Theorem}[section]
\newtheorem{lemma}{Lemma}[section]
\newtheorem{proposition}{Proposition}[section]
\newtheorem{remark}{Remark}[section]
\newtheorem{definition}{Definition}[section]
\renewcommand{\theequation}{\arabic{section}.\arabic{equation}}% 公式编号中带章节号

\begin{document}
\begin{frontmatter}



\title{A Free Boundary Problem for a Ratio-dependent Predator-prey System \tnoteref{mytitlenote}}
\tnotetext[mytitlenote]{This work was supported by \href{http://www.ctan.org/tex-archive/macros/latex/contrib/elsarticle}{NSFC 12071316}.}
\author{Lingyu Liu}
\ead{liu\_lingyu@foxmail.com}
%\address[mymainaddress,mysecondaryaddress]{Department of Mathematics, Sichuan University, Chengdu 610065, PR China}
\address{Department of Mathematics, Sichuan University, Chengdu 610065, PR China}

\begin{abstract}\large
In this paper we study a free boundary problem for a ratio-dependent predator-prey system in one space dimension, with the free boundary only caused by the prey. The long time behaviors of solution as $t\rightarrow\infty$ are discussed. Then we establish a spreading-vanishing dichotomy and the criteria for spreading and vanishing. Finally, when spreading occurs, an accurate limit of the asymptotic speed of $h(t)$ is provided as $t\rightarrow\infty$.
\end{abstract}

\begin{keyword}\large
\texttt{free boundary\sep ratio-dependent model\sep spreading-vanishing dichotomy\sep criteria \sep asymptotic speed}
\end{keyword}

\end{frontmatter}


\section{\large{Introduction}}\large
The expanding of an alien or a invasive species and the conservation of native species are discussed widely in biological mathematics.
Many mathematical models are proposed to investigate the basis of the above problems. For example, Du \cite{Du1,Du3,Du4,Du5}, Wang
\cite{wang6,wang7,wang8,wang9,wang10,wang11,wang12} and other scholars \cite{Du6,14,15,16} have studied a lot of biomathematics models with a free boundary and established many remarkable results.  The predator-prey model in a one-dimensional habitat is represented by the following system
$$
\left\{
\begin{array}{ll}
u_t-u_{xx}={\lambda} u-u^2-bf(u,v)v,    &t>0,~x>0,\\
v_t -dv_{xx}={\mu} v-v^2+cf(u,v)v,       &t>0,~x>0,
\end{array}
\right.
$$
where $u(t,x)$ and $v(t,x)$ denote the population densities of prey and predator at some time $t$ and position $x$, respectively; ${\lambda}$, ${\mu}$, $b$, $d$, $c$ are positive constants; the function $f(u,v)$ represents functional response. The classical Lotka-Volterra model assumes that $f(u,v)=u$.
As far as we know, this model displays the ``paradox of enrichment" raised by Hairston et al
(\cite{parodoxofenrichment}) which states that sufficient enrichment or increase of the prey-carrying capacity will destabilize the otherwise stable interior equilibrium. Another is the so-called ``biological control paradox" proposed by Luck (\cite{biologicalcontrolparadox}), stating that it does not exist a both low and stable prey equilibrium density. While the ratio-dependent response function $f(u,v)=\frac{u}{u+mv}$ ($m>0$) does not own both defects mentioned above. More and more evidences show that in some specific ecological environments, especially when the predators have to actively seek, share and plunder the preys, a more reasonable predator-prey model should be the ratio-dependent model(\cite{17,18,19}).

In the real world, the following phenomenons are very common to us,

(i) One kind of species (prey) inhabits in a initial region. At some time, another kind of species (the alien or invasive species, predator) invades such region.

(ii) At the initial state, a kind of pest species (prey) occupied an area (initial habitat). In order to control and eliminate such pest species, the most economical and environment-friendly approach is to use biological control, in other words, to put a certain natural enemy of pest species (predator) in this area.

Generally, both predator and prey tend to migrate outward to get a new habitat. Significantly, the prey has a stronger tendency to avoid being hunted in both phenomenons above. So it is rational that the free boundary is determined only by the prey. In this model, we suppose that the predator almost lives on this prey as a result of the features of partial eclipse, picky eaters and the restraint of external environment. In order to survive the predator should follow almost the same trajectory as prey, and so is roughly consistent with the move curve (free boundary). Assume that the left boundary is fixed and the right boundary is free, and the spreading front speed is proportional to the prey's population gradient at the front. Thus, the biometrics model in one dimensional is established as follows,
\begin{equation}\label{Q}
\left\{
\begin{array}{ll}
u_t-u_{xx}={\lambda} u-u^2-\frac{buv}{u+mv},    &t>0,~0< x<h(t),\\
v_t -dv_{xx}={\mu} v-v^2+\frac{cuv}{u+mv},       &t>0,~0<x<h(t),\\
u_x(t,0)=v_x(t,0)=0,                          &t\geq0,\\
u(t,x)=v(t,x)=0,                                  &t\geq0,~x\geq h(t),\\
h^\prime (t)=-{\rho} u_x(t,h(t)),                &t\geq0,\\
u(0,x)=u_0(x),~v(0,x)=v_0(x),      &0\leq x\leq h_0,\\
h(0)=h_0,
\end{array}
\right.
\end{equation}
where ${\rho}$, $h_0$ are given positive constants and $x=h(t)$ is the free boundary to be solved. The initial functions $u_0(x)$ and $v_0(x)$ satisfy
$$
\begin{aligned}
u_0, v_0\in {C}^2([0,h_0]),~u_0(x), v_0(x)>0 , x\in[0,h_0),\\
u^\prime_0(0)=u_0(h_0)=v^\prime_0(0)=v_0(h_0)=0.~~~~~~~~
\end{aligned}
$$

Recently, Wang and Zhang (\cite{wang9}) investigated the same spreading mechanism of the classical Lotka-Volterra type prey-predator model ($f(u,v)=u$). They discussed the asymptotic behaviors of two species and established the criteria for spreading and vanishing. In particular, they found some new phenomena; for instance, the sharp criteria for spreading and vanishing in regard to the initial habitat $h_0$ is not true in some cases. Besides, Wang (\cite{wang11}) studied a free boundary caused by a ratio of prey and predator of the classical Lotka-Volterra type prey-predator model. Wang and Zhao studied a double free boundaries in one space dimension in \cite{wang6} and they generalized these results to higher dimension and heterogeneous environment in \cite{Wangzhao2014a}. Moreover, Wang and Zhang (\cite{wang12}) studied a free boundary caused only by prey in left boundary and predator in right boundary.

Similar to the proofs of \cite[Theorem 2.1]{wang10} or \cite[Theorem 1.1]{wang9} with suitable modifications, we can prove the global existence, uniqueness and regularity of solution $(u,v,h)$ as follows.
\begin{theorem}\label{th1}
The problem {\upshape(\ref{Q})} admits a unique solution
$$
(u,v,h)\in [{C}^{\infty}(D_{\infty})]^2\times {C} ^{\infty}((0,\infty)),
$$
where $D_{\infty}=\{(t,x):t\in(0,\infty),~x\in[0,h(t)]\}$. Moreover, there exists a positive constant $K$ depending only on $\max u_0$, $\max v_0$, $\lambda$, $\mu$, $h_0^{-1}$, and $c$, such that
$$
0<u(t,x),~v(t,x)\leq K,~~~~0<h^{\prime}(t)\leq\rho K, ~~~~\forall~t>0,~0<x<h(t).
$$
\end{theorem}

This paper is organized as follows. In section 2, we study the long time behaviors of $(u,v)$ when $t\rightarrow\infty$. In section 3, we establish the spreading-vanishing dichotomy and provide the criteria for spreading and vanishing. Section 4 is devoted to an accurate limit of the asymptotic speed of $h(t)$. Finally, in section 5 we have a brief discussion.

\section{\large{Long time behaviors of $(u,v)$}}\large
In this section, we will study the long time behaviors of solution $(u,v)$. It follows from Theorem \ref{th1} that $x=h(t)$ is monotonically increasing. So $\lim\limits_{t\rightarrow\infty}h(t)=h_{\infty}\in(0,\infty]$. If $h_{\infty}<\infty$ and $\max\limits_{0\leq x\leq h(t)}u(t,\cdot)\rightarrow0$ as $t\rightarrow0$, then eventually the prey $u$ fails to establish and vanishes, which is called vanishing. If $h_{\infty}=\infty$, then the prey $u$ can successfully establish itself in the new environment, which is called spreading.
\subsection{\large{Vanishing case $h_{\infty}<\infty$}}
We first give a vital estimate like the proposition below. The proof is similar to \cite[Theorem 4.1]{wang12} and \cite[Theorem 2.1]{wang9}, so we omit it.

\begin{proposition}\label{lm2}
Let $(u,v,h)$ be a solution of problem {\upshape(\ref{Q})}. If $h_{\infty}<\infty$, then there exists a positive constant $M$ such that
\begin{equation}\label{8}
\|u(t,\cdot),~v(t,\cdot)\|_{{C}^1([0,h(t)])}\leq M,~~\forall t>1
\end{equation}
and
\begin{equation}\label{9}
\lim\limits_{t\rightarrow\infty}h^{\prime}(t)=0.
\end{equation}
\end{proposition}

According to Proposition \ref{lm2} and \cite[Proposition 3.1]{wang11}, we obtain the following theorem directly.

\begin{theorem}\label{th2}
Let $(u,v,h)$ be the solution of problem {\upshape(\ref{Q})}. If $h_{\infty}<\infty$, then
\begin{equation}\label{10}
\lim\limits_{t\rightarrow\infty}\max\limits_{0\leq x\leq h(t)}u(t,x)=0.
\end{equation}
\end{theorem}

This theorem tells us that if the prey $u$ can not spread to the whole space, then they will be vanished.

Now we discuss the long time behaviors of $v$ when $h_{\infty}<\infty$.
\begin{theorem}
Assume that $h_{\infty}<\infty$.\\
{\upshape(i)}If $h_{\infty}\leq\frac{\pi}{2}\sqrt{{d}/{\mu}}$, then
\begin{equation}\label{37}
\lim\limits_{t\rightarrow\infty}\max\limits_{0\leq x\leq h(t)}v(t,x)=0.
\end{equation}
{\upshape(ii)} If $h_{\infty}>\frac{\pi}{2}\sqrt{d/{{\mu}}}$, then
\begin{equation}\label{38}
\lim\limits_{t\rightarrow\infty}\max\limits_{0\leq x\leq h(t)}|v(t,x)-V(x)|=0,
\end{equation}
where $V(x)$ is the unique positive solution of
\begin{equation}\label{39}
\left\{
\begin{array}{ll}
-dV_{xx}=V({\mu}-V),   &0<x<h_{\infty},\\
V_x(0)=V(h_{\infty})=0.
\end{array}
\right.
\end{equation}
\end{theorem}

\begin{proof}
Case (i), $h_{\infty}\leq\frac{\pi}{2}\sqrt{{d}/{\mu}}$. Define
$$Q_{\tau}^h:=\{(t,x):t\geq \tau, 0\leq x\leq h(t)\}.$$
By (\ref{10}), for any fixed $0<\delta\ll1$, there exists $T^{\delta}>0$, such that when $(t,x)\in Q_{T^{\delta}}^h$, we have $u<\delta$.
Let $w_{\delta}(t,x)$ be the unique solution of
\begin{equation}\label{40}
\left\{
\begin{array}{ll}
w_t-d w_{xx}={\mu} w-w^2+\frac{c\delta w}{\delta+mw}, &t>T^{\delta},~0<x<h_{\infty},\\
w_x(t,0)=w(t,h_{\infty})=0,&t\geq T^{\delta}, \\
w(T^{\delta},x)=\varpi(x),&0\leq x\leq h_{\infty},
\end{array}
\right.
\end{equation}
where
$$
\varpi(x)=
\left\{
\begin{array}{ll}
v(T^{\delta},x),~~&0\leq x \leq h(T^{\delta}),\\
0,       &h(T^{\delta})<x\leq h_{\infty}.
\end{array}
\right.
$$
Then we have
\begin{equation}\label{4}
v(t,x)\leq w^{\delta}(t,x)~~in~Q_{T^{\delta}}^h
\end{equation}
by the comparison principle. Denote $\lim\limits_{\delta\rightarrow0}T^{\delta}=T^0$, and then the limiting form of problem (\ref{40}) as $\delta\rightarrow0$ becomes
\begin{equation}\label{72}
\begin{cases}
w_t-d w_{xx}={\mu} w-w^2, &t>T^{0},~0<x<h_{\infty},\\
w_x(t,0)=w(t,h_{\infty})=0,&t\geq T^{0}, \\
w(T^{0},x)=\varpi^0(x),&0\leq x\leq h_{\infty},
\end{cases}
\end{equation}
where
$$
\varpi^0(x)=
\left\{
\begin{array}{ll}
v(T^{0},x),~~&0\leq x \leq h(T^{0}),\\
0,       &h(T^{0})<x\leq h_{\infty}.
\end{array}
\right.
$$
We denote the solution of problem (\ref{72}) as $\tilde w(t,x)$. Then by the comparison principle, we have $v(t,x)\leq \tilde w(t,x)$ in $Q_{T^{0}}^h$ as $\delta\rightarrow0$.
As is well-known, if $h_{\infty}\leq\frac{\pi}{2}\sqrt{{d}/{{\mu}}}$, then problem (\ref{72}) has no positive equilibrium, which follows that $\lim\limits_{t\rightarrow\infty}\tilde w(t,x)\rightarrow0$ uniformly on any compact subset of $[0,h_{\infty})$. Thus, we have
$$
\limsup\limits_{t\rightarrow\infty}v(t,x)\leq0~~uniformly~in~any~compact~subset~of~ [0,h_{\infty}).
$$
Remember $v(t,x)\geq0$, so we have
\begin{equation}\label{3}
\lim\limits_{t\rightarrow\infty}v(t,x)=0~~uniformly~in~any~compact~subset~of~ [0,h_{\infty}).
\end{equation}

We assert (\ref{37}) is true. Otherwise there exists a constant $\sigma>0$ and a sequence $\{(t_j,x_j)\}_{j=1}^\infty$ with $0\leq x_j\leq h(t_j)$ and $t_j\rightarrow\infty$ as $j\rightarrow\infty$, such that
\begin{equation}\label{2}
v(t_j,x_j)\geq\sigma,~~j=1,2,3,....
\end{equation}
Since $0\leq x_j<h_{\infty}$, there exist a subsequence of $\{x_j\}$, denoted by itself, and $x_0\in[0,h_{\infty}]$, such that $x_j\rightarrow x_0$ as $j\rightarrow\infty$. The formulas (\ref{2}) and (\ref{3}) imply that $x_0=h_{\infty}$, i.e., $x_j-h(t_j)\rightarrow0$ as $j\rightarrow\infty$. By use of (\ref{2}) firstly and (\ref{8}) secondly, there exists $\overline x_j\in(x_j,h(t_j))$ such that
$$
\left|\frac{\sigma}{x_j-h(t_j)}\right|\leq\left|\frac{v(t_j,x_j)}{x_j-h(t_j)}\right|=\left|\frac{v(t_j,x_j)-v(t_j,h(t_j))}{x_j-h(t_j)}\right|=|v_x(t_j,\overline x_j)\leq M,
$$
which is a contradiction, for $x_j-h(t_j)\rightarrow0$. Hence, (\ref{37}) hold.

Case (ii), $h_{\infty}>\frac{\pi}{2}\sqrt{d/{{\mu}}}$. By \cite[Corollary 3.4]{kjstx23}, when $h_{\infty}>\frac{\pi}{2}\sqrt{{d}/({\mu}+c)}$, the solution $w_\delta(t,x)$ of (\ref{40}) satisfies $\lim\limits_{t\rightarrow\infty}w_{\delta}(t,x)\rightarrow V_{\delta}(x)$ uniformly on any compact subset of $[0,h_{\infty})$, where $V_{\delta}(x)$ is the unique positive solution of
\begin{equation}\label{41}
\left\{
\begin{array}{ll}
-d w_{xx}={\mu} w-w^2+\frac{c\delta w}{\delta+mw}, &0<x<h_{\infty},\\
w_x(0)=w(h_{\infty})=0. \\
\end{array}
\right.
\end{equation}
Letting $\delta\rightarrow0$, by use of the continuity of $V_{\delta}(x)$ in $\delta$, it follows that $V_{\delta}(x)\rightarrow V(x)$, where $V(x)$ is defined by (\ref{39}).
Meanwhile, the limiting form of problem (\ref{41}) as $\delta\rightarrow0$  becomes problem (\ref{39}). As is well-known, if $h_{\infty}>\frac{\pi}{2}\sqrt{{d}/{{\mu}}}$, then problem (\ref{39}) has a unique positive solution $V(x)$. Combined with (\ref{4}), when $h_{\infty}>\frac{\pi}{2}\sqrt{{d}/{\mu}}(>\frac{\pi}{2}\sqrt{{d}/{\mu+c}})$, we have
\begin{equation}\label{73}
\limsup\limits_{t\rightarrow\infty} v(t,x)\leq V(x)~uniformly~on~any~compact~subset~of~ [0,h_{\infty}).
\end{equation}

On the other hand, take $\varepsilon>0$ so small that $h_{\infty}-\varepsilon>\max\{\frac{\pi}{2}\sqrt{d/{{\mu}}},h_0\}$. So there exists $T>0$ such that $h(t)>h_{\infty}-\varepsilon$ for all $t>T$. Let $V_{\varepsilon}(t,x)$ be the unique positive solution of
$$
\left\{
\begin{array}{ll}
V_t-dV_{xx}=V({\mu}-V),   &t>T,~0<x<h_{\infty}-\varepsilon,\\
V_x(t,0)=V(t,h_{\infty}-\varepsilon)=0,&  t>T,\\
V(T,x)=v(T,x),                        &0\leq x\leq h_{\infty}-\varepsilon.
\end{array}
\right.
$$
By the comparison principle we have $v(t,x)\geq V_{\varepsilon}(t,x)$ for $t>T$ and $x\in[0,h_{\infty}-\varepsilon]$. Since $h_{\infty}-\varepsilon>\frac{\pi}{2}\sqrt{d/{{\mu}}}$, it is well known that $\lim\limits_{t\rightarrow\infty}V_{\varepsilon}(t,x)=V_{\varepsilon}(x)$ uniformly on $[0,h_{\infty}-\varepsilon]$, where $V_{\varepsilon}(x)$ is the unique positive solution of
$$
\left\{
\begin{array}{ll}
-dV_{xx}=V({\mu}-V),   &0<x<h_{\infty}-\varepsilon,\\
V_x(0)=V(h_{\infty}-\varepsilon)=0.
\end{array}
\right.
$$
Therefore, we have $\liminf\limits_{t\rightarrow\infty}v(t,x)\geq V_{\varepsilon}(x)$ uniformly on $[0,h_{\infty}-\varepsilon]$. By using the continuity of $V_{\varepsilon}$ in $\varepsilon$ and letting $\varepsilon\rightarrow 0$, we get
\begin{equation}\label{42}
\liminf\limits_{t\rightarrow\infty}v(t,x)\geq V(x)~~uniformly~on~any~compact~subset~of~ [0,h_{\infty}),
\end{equation}
where $V(x)$ is defined by (\ref{39}). Combined (\ref{42}) with (\ref{73}) deduces that
\begin{equation}\label{60}
\lim\limits_{t\rightarrow\infty} v(t,x)=V(x)~~uniformly~on~any~compact~subset~of~ [0,h_{\infty}).
\end{equation}

Now we prove (\ref{38}). On the contrary we assume that (\ref{38}) is not true. Then there exists a constant $\sigma>0$ and a sequence $\{(t_j,x_j)\}_{j=1}^{\infty}$ with $0\leq x_j<h(t_j)$ and $t_j\rightarrow\infty$ as $j\rightarrow\infty$, such that
\begin{equation}\label{59}
|v(t_j,x_j)-V(x_j)|\geq 2\sigma,~~~~j=1,2,...
\end{equation}
For $0\leq x_j<h_{\infty}$, there exists a subsequence of $\{x_j\}$, denoted by itself, and $x_0\in[0,h_{\infty}]$, such that $x_j\rightarrow x_0$ as $j\rightarrow\infty$. It follows from (\ref{60}) and (\ref{59}) that $x_0=h_{\infty}$, i.e., $x_j-h(t_j)\rightarrow 0$ as $j\rightarrow\infty$. Since $V(x)$ is continuous in $[0,h_{\infty}]$, $V(h_{\infty})=0$ and $x_j\rightarrow\infty$, by (\ref{59}) we have
$|v(t_j,x_j)|\geq\sigma$ for all $j>1$. Similar to the proof of (\ref{37}), we can get a contradiction.
\end{proof}

\subsection{\large{Spreading case $h_{\infty}=\infty$}}

\begin{theorem}\label{th3}
Assume that $m\lambda>b$ and $h_{\infty}=\infty$. Then the solution $(u,v)$ of problem (\ref{Q}) satisfies
\begin{equation}\label{70}
\left\{
\begin{array}{ll}
{\lambda} -u-\frac{bv}{u+mv}=0,\\
{\mu} -v+\frac{cu}{u+mv}=0.
\end{array}
\right.
\end{equation}
Furthermore, if $0<m{\lambda}-b<b{\mu}/c$, then
$$\lim\limits_{t\rightarrow\infty}u(t,x)=u^*:=\frac{A+\sqrt{\Delta_1}}{2(b +c m^2)},~\lim\limits_{t\rightarrow\infty}v(t,x)=v^*:=\frac{u^*({\lambda}-u^*)}{b-m({\lambda}-u^*)},$$
where $A={\lambda}(2cm^2+b)-mb({\mu}+2c)$,
$\Delta_1=A^2+4(b+cm^2)[b({\mu}+c)-mc{\lambda}](m{\lambda}-b)$.
\end{theorem}

\begin{proof}
The proof uses the iteration method.

Step 1.~~For any fixed $L\gg1$ and $0<\varepsilon\ll1$, let $l_{\varepsilon}$ be given by Proposition \ref{p1} with $d=1,~\beta={\lambda}$ and $f(u)\equiv1$. Taking accounting to $h_{\infty}=\infty$, we can find $T_1>0$ such that $h(t)\geq l_{\varepsilon}$ when $t>T_1$. In view of $u\geq0,v\geq0$, we have
$$
\left\{
\begin{array}{ll}
u_t-u_{xx}\leq u({\lambda}-u),   &t>T_1,~0<x<l_{\varepsilon},\\
u_x(t,0)=0,~u(t,l_{\varepsilon})\leq K,&t>T_1,
\end{array}
\right.
$$
where $K$ defined by Theorem \ref{th1}.
Since $u(T_1,x)>0$ for $[0,l_{\varepsilon}]$, by using Proposition \ref{p1}, we have
$$
\limsup\limits_{t\rightarrow\infty}u(t,x)\leq{\lambda}+\varepsilon ~uniformly~on~[0,L].
$$
The arbitrariness of $\varepsilon$ and $L$ implies that
\begin{equation}\label{17}
\limsup\limits_{t\rightarrow\infty}u(t,x)\leq{\lambda}:=\overline u_1
\end{equation}
 uniformly on any compact subset of $[0,\infty)$.

Step 2.~~For any fixed $L\gg1$, $0<\delta\ll1$ and $0<\varepsilon\ll1$, let $l_{\varepsilon}$ be given by Proposition \ref{p1} with $\beta=v_1^{\delta}$ and $f(v)={m(v-v_2^{\delta})}/[{(\overline u_1+\delta)+m v}]$. Thanks to (\ref{17}), there exists $T_2>T_1$ such that $u(t,x)<\overline u_1+\delta$ for $t>T_2$. Thus
$$
\begin{array}{ll}
v_t-dv_{xx}&\leq {\mu} v-v^2+\frac{c(\overline u_1+\delta)v}{(\overline u_1+\delta)+mv}\\
           &=-v\frac{mv^2-[m{\mu}+(c-1)(\overline u_1+\delta)]v-(\overline u_1+\delta)}{(\overline u_1+\delta)+mv}\\
           &=mv\frac{(v-v_2^{\delta})(v_1^{\delta}-v)}{(\overline u_1+\delta)+m v},
\end{array}
$$
and
$$
v_x(t,0)=0,~v(t,l_{\varepsilon})\leq K, ~~\forall~t>T_2,
$$
where
$$
v_1^{\delta}=\frac{{m{\mu}+(c-1)(\overline u_1+\delta)+\sqrt{[m{\mu}+(c-1)(\overline u_1+\delta)]^2+4m(\overline u_1+\delta){\mu}}}}{2m}>0,
$$

$$
v_2^{\delta}=\frac{{m{\mu}+(c-1)(\overline u_1+\delta)-\sqrt{[m{\mu}+(c-1)(\overline u_1+\delta)]^2+4m(\overline u_1+\delta){\mu}}}}{2m}<0.
$$

According to $v(T_2,x)>0$ for $x\in[0,l_{\varepsilon}]$ and Proposition \ref{p1}, we have $v(t,x)<v_1^{\delta}+\varepsilon$ uniformly for $t>T_2$ and $x\in[0,L]$.
The arbitrariness of $\varepsilon$, $\delta$ and $L$ implies that
\begin{equation}\label{1.9}
\limsup\limits_{t\rightarrow\infty}v(t,x)\leq v_1^0:=\overline v_1
\end{equation}
uniformly on any compact subset of $[0,\infty)$, where
$$
\overline v_1=\frac{{m{\mu}+(c-1)\overline u_1+\sqrt{[m{\mu}+(c-1)\overline u_1]^2+4m{\mu}\overline u_1}}}{2m}>0.
$$

Step 3.  For given $L\gg1$, $0<\delta\ll1$ and $0<\varepsilon\ll1$, let $l_{\varepsilon}$ be given by Proposition \ref{p1} with $d=1$, $f(u)={(u-u_2^{\delta})}/[{u+m(\overline v_1+\delta)}]$ and $\beta=\underline u_1^{\delta}$. By (\ref{1.9}), there exists $T_3>T_2$ such that $v(t,x)<\overline v_1+\delta$. Thus
$$
\begin{array}{ll}
u_t-u_{xx}&\geq {\lambda} u-u^2-\frac{bu(\overline v_1+\delta)}{u+m(\overline v_1+\delta)}\\
          &=-u\frac{u^2-[{\lambda}-m(\overline v_1+\delta)]u-(m{\lambda}-b)(\overline v_1+\delta)}{u+m(\overline v_1+\delta)}\\
          &=u\frac{(u- u_2^{\delta})( u_1^{\delta}-u)}{u+m(\overline v_1+\delta)}
\end{array}
$$
and
$$u_x(t,0)=0,~u(t,l_{\varepsilon})\geq0,~t>T_3,$$
where
$$
 u_1^{\delta}=\frac{{\lambda}-m(\overline v_1+\delta)+\sqrt{[{\lambda}-m(\overline v_1+\delta)]^2+4(m{\lambda}-b)(\overline v_1+\delta)}}{2}>0,
$$
$$
 u_2^{\delta}=\frac{{\lambda}-m(\overline v_1+\delta)-\sqrt{[{\lambda}-m(\overline v_1+\delta)]^2+4(m{\lambda}-b)(\overline v_1+\delta)}}{2}<0.
$$

According to $u(T_3,x)>0$ for $x\in[0,l_{\varepsilon}]$ and Proposition \ref{p1}, it follows that $u(t,x)>\underline u_1^{\delta}-\varepsilon$ uniformly for $t>T_3$ and $x\in[0,L]$.
The arbitrariness of $\varepsilon$, $\delta$ and $L$ implies that
\begin{equation}\label{20}
\liminf\limits_{t\rightarrow\infty}u(t,x)\geq u_1^0:=\underline u_1>0.
\end{equation}
uniformly on any compact subset of $[0,\infty)$, where
$$
\underline u_1=\frac{{\lambda}-m\overline v_1+\sqrt{({\lambda}-m\overline v_1)^2+4(m{\lambda}-b)\overline v_1}}{2}.
$$

step 4.  For given $L\gg1$, $0<\delta\ll1$ and $0<\varepsilon\ll1$, let $l_{\varepsilon}$ be given by Proposition \ref{p1} with $\beta=v_3^{\delta}$ and $f(v)={m(v-v_4^{\delta})}/[{(\underline u_1-\delta)+m v}]$. By use of (\ref{20}), there exists $T_4>T_3$ such that $u(t,x)>\underline u_1-\delta$ for $t>T_4$. Thus
$$
\begin{array}{ll}
v_t-dv_{xx}&\geq {\mu} v-v^2+\frac{c(\underline u_1-\delta)v}{(\underline u_1-\delta)+m v}\\
           &=-v\frac{mv^2-[m{\mu}+(c-1)(\underline u_1-\delta)]v-(\underline u_1-\delta){\mu}}{(\underline u_1-\delta)+m v}\\
           &=mv\frac{(v-v_4^{\delta})(v_3^{\delta}-v)}{(\underline u_1-\delta)+m v},
\end{array}
$$
and
$$
v_x(t,0)=0,~v(t,l_{\varepsilon})\geq0, ~~\forall~t>T_4,
$$
where
$$
v_3^{\delta}=\frac{{m{\mu}+(c-1)(\underline u_1-\delta)+\sqrt{[m{\mu}+(c-1)(\underline u_1-\delta)]^2+4m{\mu}(\underline u_1-\delta)}}}{2m}>0,
$$

$$
v_4^{\delta}=\frac{{m{\mu}+(c-1)(\underline u_1-\delta)-\sqrt{[m{\mu}+(c-1)(\underline u_1-\delta)]^2+4m{\mu}(\underline u_1-\delta)}}}{2m}<0.
$$
Similar to Step 3, we have
\begin{equation}\label{19}
\liminf\limits_{t\rightarrow\infty}v(t,x)\geq v_3^0:=\underline v_1
\end{equation}
uniformly on any compact subset of $[0,\infty)$, where
$$
\underline v_1=\frac{{m{\mu}+(c-1)\underline u_1+\sqrt{[m{\mu}+(c-1)\underline u_1]^2+4m{\mu}\underline u_1}}}{2m}>0.
$$

Step 5.~~For given $L\gg1$, $0<\delta\ll1$ and $0<\varepsilon\ll1$, let $l_{\varepsilon}$ be given by Proposition \ref{p1} with $\beta=\underline u_3^{\delta}$ and $f(u)={(u-u_4^{\delta})}/[{u+m(\underline v_1-\delta)}]$. By (\ref{19}), there exists $T_5>T_4$ such that $v(t,x)>\underline v_1-\delta$ for $t>T_5$. Thus
$$
\begin{array}{ll}
u_t-u_{xx}&\geq {\lambda} u-u^2-\frac{bu(\underline v_1-\delta)}{u+m(\underline v_1-\delta)}\\
          &=-u\frac{u^2-[{\lambda}-m(\underline v_1-\delta)]u-(m{\lambda}-b)(\underline v_1-\delta)}{u+m(\underline v_1-\delta)}\\
          &=u\frac{(u- u_4^{\delta})( u_3^{\delta}-u)}{u+m(\underline v_1-\delta)}
\end{array}
$$
and
$$u_x(t,0)=0,~u(t,l_{\varepsilon})\geq0,~t>T_5,$$
where
$$
 u_3^{\delta}=\frac{{\lambda}-m(\underline v_1-\delta)+\sqrt{[{\lambda}-m(\underline v_1-\delta)]^2+4(m{\lambda}-b)(\underline v_1-\delta)}}{2}>0,
$$
$$
 u_4^{\delta}=\frac{{\lambda}-m(\underline v_1-\delta)-\sqrt{[{\lambda}-m(\underline v_1-\delta)]^2+4(m{\lambda}-b)(\underline v_1-\delta)}}{2}<0.
$$
Similar to Step 2, we have
\begin{equation}\label{20}
\limsup\limits_{t\rightarrow\infty}u(t,x)\leq u_3^0:=\overline u_2>0
\end{equation}
uniformly on any compact subset of $[0,\infty)$, where
$$
\overline u_2=\frac{{\lambda}-m\underline v_1+\sqrt{({\lambda}-m\underline v_1)^2+4(m{\lambda}-b)\underline v_1}}{2}.
$$

Repeating the above processes, we can find four sequences $\{\overline u_i\}$, $\{\overline v_i\}$, $\{\underline u_i\}$, $\{\underline v_i\}$ such that for all $i$,
$$
\underline u_i\leq\liminf\limits_{t\rightarrow\infty}u(t,x)\leq\limsup\limits_{t\rightarrow\infty}u(t,x)\leq\overline u_i,
$$
$$
\underline v_i\leq\liminf\limits_{t\rightarrow\infty}v(t,x)\leq\limsup\limits_{t\rightarrow\infty}v(t,x)\leq\overline v_i
$$
uniformly on any compact subset of $[0,\infty)$. Moreover, for $s>0$, denote
$$
\phi(s)=\frac{{\lambda}-ms+\sqrt{({\lambda}-ms)^2+4(m{\lambda}-b)s}}{2},$$
$$
\psi(s)=\frac{{m{\mu}+(c-1)s+\sqrt{[m{\mu}+(c-1)s]^2+4m{\mu} s}}}{2m},
$$
then
$$
\overline v_i=\psi(\overline u_i),~\underline u_i=\phi(\overline v_i),~\underline v_i=\psi(\underline u_i),~\overline u_{i+1}=\phi(\underline v_i),~~~i=1,2,3,...
$$
And the sequences $\{\overline u_i\}$, $\{\overline v_i\}$, $\{\underline u_i\}$, $\{\underline v_i\}$ satisfy
$$
\underline {u}_1\leq\cdots\leq\underline u_i\leq\liminf\limits_{t\rightarrow\infty} u(t,x)\leq\limsup\limits_{t\rightarrow\infty} u(t,x)\leq\cdots\leq\overline u_i\leq\cdots\leq\overline u_1,
$$
$$
\underline v_1\leq\cdots\leq\underline v_i\leq\liminf\limits_{t\rightarrow\infty} v(t,x)\leq\limsup\limits_{t\rightarrow\infty} v(t,x)\leq\cdots\leq\overline v_i\leq\cdots\leq\overline v_1,
$$
which indicates that $\{\overline u_i\}$ and $\{\overline v_i\}$ are monotonically non-increasing with lower boundaries and $\{\underline u_i\}$ and $\{\underline v_i\}$ are monotonically non-decreasing with upper boundaries. Then $\{\overline u_i\}$, $\{\overline v_i\}$, $\{\underline u_i\}$, $\{\underline v_i\}$ have limits denoted by $\overline u$, $\overline v$, $\underline u$, $\underline v$, respectively, as $i\rightarrow\infty$. Furthermore, they satisfy
$$
\overline v=\psi(\overline u),~\underline u=\phi(\overline v),~\underline v=\psi(\underline u),~\overline u=\phi(\underline v),
$$
which follows that
\begin{equation}\label{24}
{\lambda}-\underline u-\frac{b\overline v}{\underline u+m\overline v}=0,~{\lambda}-\overline u-\frac{b\underline v}{\overline u+m\underline v}=0,
\end{equation}
\begin{equation}\label{25}
{\mu}-\overline v+\frac{c\overline u}{\overline u+m\overline v}=0,~{\mu}-\underline v+\frac{c\underline u}{\underline u+m\underline v}=0.
\end{equation}
So $\overline u=\underline u=:u^*$, $\overline v=\underline v=:v^*$, and $(u^*,v^*)$ satisfies (\ref{70}).

By the first equation of (\ref{70}), we have
\begin{equation}\label{71}
v=\frac{u({\lambda}-u)}{b-m({\lambda}-u)}.
\end{equation}
Substituting (\ref{71}) into the second equation of (\ref{70}), we have
$$
(b+cm^2)u^2-[{\lambda}(2cm^2+b)-mb({\mu}+2c)]u-[b({\mu}+c)-mc{\lambda}](m{\lambda}-b)=0.
$$
The assumption $m{\lambda}-b<b{\mu}/c$ implies $b({\mu}+c)-mc{\lambda}>0$.
Then $$u_1=\frac{A+\sqrt{\Delta_1}}{2(b+cm^2)}>0,~u_2=\frac{A-\sqrt{\Delta_1}}{2(b+cm^2)}<0.$$
So $u^*=u_1$ and $v^*=\frac{u^*({\lambda}-u^*)}{b-m({\lambda}-u^*)}$.
Thus, when $0<m{\lambda}-b<b{\mu}/c$, we get
$$
(u^*,v^*):=(\frac{A+\sqrt{\Delta_1}}{2(b +c m^2)},\frac{u^*({\lambda}-u^*)}{b-m({\lambda}-u^*)}).
$$
\end{proof}

\section{\large{The spreading-vanishing dichotomy and the criteria of spreading and vanishing}}\large
In this section, we want to establish the spreading-vanishing dichotomy and study the criteria of spreading and vanishing of problem (\ref{Q}). We first study an eigenvalue problem.

For any given $l>0$, let $\sigma_1(l)$ be the first eigenvalue of
\begin{equation}\label{81}
\left\{
\begin{array}{ll}
-\phi_{xx}-({\lambda}-\frac{b}{m})\phi=\sigma\phi, ~~0<x<l,\\
\phi_x(0)=0=\phi(l).
\end{array}
\right.
\end{equation}

\begin{lemma}\label{lm4.1}
Assume that $m{\lambda}>b$. If $h_{\infty}<\infty$, then $\sigma_1(h_{\infty})\geq0$.
\end{lemma}
\begin{proof}
We assume $\sigma_1(h_{\infty})<0$ to get a contradiction. By the continuity of $\sigma_1(l)<0$ in $l$ and $h(t)$ in t, there exists $T\geq1$ such that $\sigma_1(h{(T)})<0$.
Suppose $w(t,x)$ be the unique solution of
$$
\left\{
\begin{array}{ll}
w_t-w_{xx}=w({\lambda}-\frac{b}{m}-w),    &t>T,~0<x<h(T),\\
w_x(t,0)=w(t,h(T))=0,  &t\geq T,\\
w(T,x)=u(T,x),&0\leq x\leq h(T).
\end{array}
\right.
$$
By using the compare principle, we have
$$
w(t,x)\leq u(t,x),~~~t\geq T,~0\leq x\leq h(T).
$$
In view of $\sigma_1(h(T))<0$, it is well known that $w(t,x)\rightarrow w^*(x)$ uniformly in any compact subset of $[0,h(T))$ as $t\rightarrow\infty$, where $w^*$ is the unique positive solution of
\begin{equation}
\left\{
\begin{array}{ll}
-w^*_{xx}=w^*({\lambda}-\frac{b}{m}-w^*),    &0<x<h(T),\\
w^*_x(0)=w^*(h(T))=0,
\end{array}
\right.
\end{equation}
Thus, $\liminf\limits_{t\rightarrow\infty}u(t,x)\geq\lim\limits_{t\rightarrow\infty}w(t,x)=w^*(x)>0$, which contradicts (\ref{10}). The proof is completed.
\end{proof}

Define
$$
\Lambda:=\frac{\pi}{2}\sqrt{\frac{m}{m{\lambda}-b}}.
$$

\begin{theorem}\label{th4.1}
Assume that $m{\lambda}>b$. \\
{\upshape(i)}If $h_{\infty}<\infty$, then $h_{\infty}\leq\Lambda$;\\
{\upshape(ii)}If $h_0\geq\Lambda$, then $h_{\infty}=\infty$ for all ${\rho}>0$.
\end{theorem}
\begin{proof}
We assert that (i) is true. Otherwise $\Lambda <h_{\infty}<\infty$.
It is well known that problem (\ref{81}) can be solved explicitly in terms of exponentials by the principle eigenvalue $\sigma_1(l)=-({\lambda}-\frac{b}{m})+(\frac{\pi}{2l})^2$, and in this case $\phi_1(l)=\cos\frac{\pi x}{2l}$. It is follows from our assumption $h_{\infty}>\Lambda$ that $\sigma_1(h_{\infty})<0$, which contradicts Lemma \ref{lm4.1}. In consequence, we get either $h_{\infty}\leq\Lambda$ or $h_{\infty}=\infty$.

Accordingly, $h_0\geq\Lambda$ implies $h_{\infty}=\infty$ for all ${\rho}>0$. Thus (ii) is proved.
\end{proof}

To emphasize the dependence of $h$ on $\rho$ and $h_0$, we substitute $h(t)$ with $h(\rho,h_0;t)$.
\begin{lemma}\label{lm4}
Suppose $m{\lambda}>b$. For any given $h_0>0$, there exists ${\rho}^0$ depending on $u_0$, $v_0$ and $h_0$, such that $h({\rho},h_0;\infty) =\infty$ for all ${\rho}>{\rho}^0$.
\end{lemma}

Lemma \ref{lm4} can be deduced by Theorem \ref{th1} and \cite[Lemma 3.2]{wang9} immediately, so we omit the proof.

Define
$$
\mathcal E=\{k>0:h({\rho},h_0;\infty)=\infty,~\forall~h_0\geq k,\forall~{\rho}>0\},
$$
$$
\mathcal F=\{\ k>0:~\forall~0<h_0<k,~\exists~{\rho}_0>0,
~ s.t.~h({\rho},h_0;\infty)<\infty,~\forall ~ 0<{\rho}\leq{\rho}_0\}.
$$

Assume $m{\lambda}>b$ and set
$$h^*:=\inf\mathcal E,~h_*:=\sup\mathcal {F}.$$
It is natural that $h_*\leq h^*$. Theorem \ref{th4.1} implies that $[\Lambda,\infty)\subset\mathcal E$. Thus, we have $h^*\leq{\Lambda}$. In addition, we draw an important conclusion as the following lemma shows.

\begin{lemma}\label{lm4.5}
$h_*\geq\frac{\pi}{2}{\lambda}^{-\frac{1}{2}}$.
\end{lemma}

\begin{proof}
We assert that if $h_0<\frac{\pi}{2}{\lambda}^{-\frac{1}{2}}$, then there exists ${\rho}_0>0$ such that $h({\rho},h_0;\infty)<\infty$ for all $0<{\rho}\leq{\rho}_0$. Thus, we have $\frac{\pi}{2}{\lambda}^{-\frac{1}{2}}\in\mathcal{F}$, which implies that $h_*\geq\frac{\pi}{2}{\lambda}^{-\frac{1}{2}}$. We use the method in reference \citep{Du2} to construct a suitable upper solution of (\ref{Q}).
Define
$$
\sigma(t)=h_0(1+\delta-\frac{\delta}{2}e^{-\gamma t}),~t\geq0;~V(y)=\cos\frac{\pi y}{2},~0\leq y\leq1,
$$
and
$$
w(t,x)=Ce^{-\alpha t}V(\frac{x}{\sigma(t)}),~t\geq0,~0\leq x\leq\sigma (t),
$$
where $\delta$, $\gamma$, $\alpha$, $C$ are positive constants to be determined. We choose $C$ sufficiently large such that $u_0(x)\leq C\cos(\frac{\pi}{2}\frac{x}{h_0(1+{\delta}/{2})})$. Besides, it is easy to verify that
$$u_0(x)\leq w(0,x),~h_0<(1+\frac{\delta}{2})h_0=\sigma(0),~0<x<\sigma(t),$$
$$w_x(t,0)=w(t,\sigma(t))=0,~t>0.$$
Noting that $\sigma(t)<h_0(1+\delta)$ for all $t>0$, direct calculations deduce that
$$
\begin{array}{ll}
&w_t-w_{xx}-w({\lambda}-w)\\
=&w[-\alpha+(\frac{\pi}{2})x\sigma^{-2}\sigma^{\prime}\tan(\frac{\pi}{2}\frac{x}{\sigma(t)})+(\frac{\pi}{2})^2\sigma^{-2}(t)-{\lambda}+w]\\
\geq& w[-\alpha+(\frac{\pi}{2})^2(1+\delta)^{-2}h_0^{-2}-{\lambda}+w].\\
\end{array}
$$
Note that ${\lambda}<(\frac{\pi}{2})^2h_0^{-2}$, and we can find $\delta>0$ such that
$$
(\frac{\pi}{2})^2(1+\delta)^{-2}h_0^{-2}-{\lambda}=\frac{1}{2}\left[(\frac{\pi}{2})^2h_0^{-2}-{\lambda}\right].
$$
Take $\alpha=\frac{1}{2}\left[(\frac{\pi}{2})^2h_0^{-2}-{\lambda}\right]$, and we have $w_t-w_{xx}-w({\lambda}-w)\geq0$ for $t>0$ and $0<x<\sigma(t)$. Let ${\rho}_0={\rho}_0(C):=\frac{\delta\gamma h_0^2}{C\pi}$ and $\alpha=\gamma$, then for any $0<{\rho}\leq{\rho}_0$, we have
$$
\sigma^{\prime}(t)\geq-{\rho} w_x(t,\sigma(t)),~t\geq0.
$$
By applying Proposition \ref{rm1}, we have $u(t,x)\leq w(t,x)$ and $h(t)\leq\sigma(t)$ for $t>0$ and $0\leq x\leq \sigma(t)$. Thus $\sigma(t)\rightarrow h_0(1+\delta)$ as $t\rightarrow\infty$, implying that $h({\rho},h_0;\infty)<\infty$. The proof is finished.
\end{proof}

From the above discussions, we have the spreading-vanishing dichotomy and the criteria for spreading and vanishing.
\begin{theorem}\label{th4.2}
Assume that $m{\lambda}>b$. Then either spreading ($h_{\infty}=\infty$) or vanishing ($h_{\infty}\leq\Lambda$) holds. To be more precisely,

{\upshape (i)}$h({\rho},h_0;\infty)=\infty$ for all $h_0> h^*$ and ${\rho}>0$;

{\upshape (ii)}for any given $h_0>0$, there exists ${\rho}^0>0$, which depends on $u_0$, $v_0$, $h_0$, such that $h({\rho},h_0;\infty)=\infty$ for all ${\rho}>{\rho}^0$;

{\upshape (iii)}for any given $0<h_0<h_*$, there exists ${\rho}_0>0$, which also depends on $u_0$, $v_0$, $h_0$, such that $h({\rho},h_0;\infty)\leq\Lambda$ for all $0<{\rho}\leq{\rho}_0$.
\end{theorem}

\begin{remark}
The estimates $\frac{\pi}{2}{\lambda}^{-\frac{1}{2}}\leq h_*\leq h^*\leq\frac{\pi}{2}\sqrt{\frac{m}{m{\lambda}-b}}$ are obtained. But
we can not prove $h^*=h_*$, and get the sharp criteria for spreading and vanishing with respect to the initial $h_0$.
\end{remark}

\section{\large{Asymptotic speed of $h$}}\large
In this section, we mainly give some estimates of ${h(t)}/{t}$ as $t\rightarrow\infty$. In the following we always assume that $m\lambda>b$.
\begin{definition}
Assume that $u(t,x)$ is a nonnegative function for $x>0$, $t>0$. We call $c_*$ as the asymptotic speed of $u(t,x)$ if

{\upshape(a)}$\lim\limits_{t\rightarrow\infty}\sup\limits_{x>(c_*+\varepsilon)t}u(t,x)=0$ for any given $\varepsilon>0$,

{\upshape(b)}$\lim\limits_{t\rightarrow\infty}\inf\limits_{0<x<(c_*-\varepsilon)t}u(t,x)>0$ for any given $\varepsilon\in(0,c_*)$.
\end{definition}

For the following diffusive logistic equations
\begin{equation}\label{78}
\begin{cases}
w_t-dw_{xx}=w(a-bw),  &t>0,~x>0,\\
w_x(t,0)=0,&t\geq0,\\
w(0,x)=w_0>0,&x\geq0.
\end{cases}
\end{equation}

\begin{proposition}\upshape{(\citep{24})}\label{p6}
It is well known that the asymptotic speed of problem (\ref{78}) $c_*=2\sqrt{ad}$ and
$$
\liminf\limits_{t\rightarrow\infty}\inf\limits_{0<x<(c_*-\varepsilon)t}w(t,x)=\frac{a}{b},~~\limsup\limits_{t\rightarrow\infty}\sup\limits_{x>(c_*+\varepsilon)t}w(t,x)=0
$$
for any small $\varepsilon\in(0,c_*)$.
\end{proposition}

For the classical logistic problem with a free boundary problem
\begin{equation}\label{79}
\begin{cases}
w_t-dw_{xx}=w(a-bw),~&t>0,~0<x<s(t),\\
s^{\prime}(t)=-{\rho} w_x(t,s(t)),&t\geq0,\\
w_x(t,0)=w(t,s(t))=0, &t\geq0,\\
w(0,x)=w_0,~s(0)=s_0,&0\leq x\leq s_0.
\end{cases}
\end{equation}
It has proved that the expanding front $s(t)$ moves at a constant speed for large time (see \cite{Du1,Du6}), i.e.,
$$
s(t)=(c_0+o(1))t~~as ~~t\rightarrow\infty.
$$
And $c_0$ is determined by the following auxiliary elliptic problem
\begin{equation}\label{63}
\left\{
\begin{array}{ll}
dq^{\prime\prime}-cq^{\prime}+q(a-bq)=0,~~0<y<\infty,\\
q(0)=0,~~q^{\prime}(0)=c/{\rho},~~q(\infty)=\frac{a}{b},\\
c\in(0,2\sqrt{ad});~~q^{\prime}(y)>0,~~0<y<\infty,
\end{array}
\right.
\end{equation}
where ${\rho}$, $d$, $a$, $b$ are positive constants.
\begin{proposition}\label{p3}\upshape{(\citep{Du6})}
The problem (\ref{63}) has a unique solution $(q(y),c)$ and $c({\rho},d,a,b)$ is strictly increasing in ${\rho}$ and $a$, respectively. Moreover,
\begin{equation}\label{64}
\lim\limits_{\frac{a{\rho}}{bd}\rightarrow\infty}\frac{c({\rho},d,a,b)}{\sqrt{ad}}=2,~~\lim\limits_{\frac{a{\rho}}{bd}\rightarrow 0}\frac{c({\rho},d,a,b)}{\sqrt{ad}}\frac{bd}{a{\rho}}=\frac{1}{\sqrt{3}}.
\end{equation}
\end{proposition}

For $t\geq 0$ and $x\geq 0$, the following inequalities are natural,
\begin{equation}\label{65}
\left\{
\begin{array}{ll}
({\lambda}-{b}/{m})u-u^2\leq{\lambda} u-u^2-\frac{buv}{u+mv}\leq{\lambda} u-u^2,\\
{\mu} v-v^2\leq {\mu} v -v^2+\frac{cuv}{u+mv}\leq({\mu}+c)v-v^2.
\end{array}
\right.
\end{equation}

Define
$$
c_1=2\sqrt{{\lambda}-{b}/{m}},~c_2=2\sqrt{{\lambda}},~c_3=2\sqrt{d{\mu}},~c_4=2\sqrt{d({\mu}+c)}.
$$
\begin{theorem}\label{th6.1}
For any given $0<\varepsilon\ll1$, the following conclusions hold,
\begin{equation}\label{76}
\limsup\limits_{{\rho}\rightarrow\infty}\limsup\limits_{t\rightarrow\infty}\frac{h(t)}{t}\leq c_2,~\liminf\limits_{{\rho}\rightarrow\infty}\liminf\limits_{t\rightarrow\infty}\frac{h(t)}{t}\geq c_1.
\end{equation}
Moreover, when ${\rho}\rightarrow\infty$,
\begin{equation}\label{74}
\limsup\limits_{t\rightarrow\infty}\sup\limits_{x>(c_2+\varepsilon)t}u(t,x)=0,
\end{equation}
\begin{equation}\label{74.0}
\liminf\limits_{t\rightarrow\infty}\inf\limits_{0<x<(c_1-\varepsilon)t}u(t,x)\geq{\lambda}-\frac{b}{m},
\end{equation}

\begin{equation}\label{75}
\limsup\limits_{t\rightarrow\infty}\sup\limits_{x>(c_4+\varepsilon)t}v(t,x)=0,
\end{equation}
\begin{equation}\label{75.0}
\liminf\limits_{t\rightarrow\infty}\inf\limits_{0<x<(c_3-\varepsilon)t}v(t,x)\geq{\mu}.
\end{equation}
\end{theorem}
\begin{proof}
It follows from the first limit of (\ref{64}) that
$$
\lim_{{\rho}\rightarrow\infty}c({\rho},1,{\lambda}-b/m,1)=c_1,~\lim_{{\rho}\rightarrow\infty}c({\rho},1,{\lambda},1)=c_2.
$$
In view of (\ref{65}) and Proposition \ref{rm1}, it can be deduced that for any given $\varepsilon>0$, there exists ${\rho}_{\varepsilon}\gg1$ such that, for all ${\rho}>{\rho}_{\varepsilon}$,
$$
 c_1-\frac{\varepsilon}{2}<c({\rho},1,{\lambda}-b/m,1)\leq\liminf\limits_{t\rightarrow\infty}\frac{h(t)}{t},~
\limsup\limits_{t\rightarrow\infty}\frac{h(t)}{t}\leq c({\rho},1,{\lambda},1)<c_2+\frac{\varepsilon}{2}.
$$
By letting $\varepsilon\rightarrow0$, (\ref{76}) holds.

Besides, we can find $\tau_1\gg1$ such that for all $t\geq\tau_1$ and ${\rho}\gg{\rho}_{\varepsilon}$,
\begin{equation}\label{80}
(c_1-\varepsilon)t<h(t)<(c_2+\varepsilon)t.
\end{equation}
Obviously, (\ref{74}) holds.

The proof of (\ref{74.0}). Let $(\underline u,\underline h)$ be the unique solution of (\ref{79}) with $({\rho},d,a,b)=({\rho},1,{\lambda}-b/m,1)$.
Then $h(t)\geq\underline h(t)$, $u(t,x)\geq\underline u(t,x)$ for $t>0$ and $0<x\leq\underline h(t)$ by Proposition \ref{rm1}. Make use of \cite[Theorem 3.1]{wang5}, we get
$$
\lim_{t\rightarrow\infty}(\underline h(t)-c^*t)=H\in\mathbb{R},~\lim_{t\rightarrow\infty}||\underline u(t,x)-q^*(c^*t+H-x)||_{L^{\infty}([0,\underline h(t)])}=0,
$$
where $(q^*(y),c^*)$ is the unique solution of (\ref{63}) with $({\rho},d,a,b)=({\rho},1,{\lambda}-{b}/{m},1)$, i.e., $c^*=c({\rho},1,{\lambda}-b/m,1)>c_1-\frac{\varepsilon}{2}$. Obviously, $\min_{[0,(c_1-\frac{\varepsilon}{2})t]}(c^*t+H-x)\rightarrow\infty$ as $t\rightarrow\infty$.
Owing to $q^*(y)\nearrow{\lambda}-{b}/{m}$ as $y\nearrow\infty$, we have $\min_{[0,(c_1-\frac{\varepsilon}{2})t]}q^*(c^*t+H-x)\rightarrow{\lambda}-b/m$ as $t\rightarrow\infty$. So $\min_{[0,(c_1-\frac{\varepsilon}{2})t]}\underline u(t,x)\rightarrow{\lambda}-b/m$ as $t\rightarrow\infty$. Thus (\ref{74.0}) holds due to $u(t,x)\geq\underline u(t,x)$ for $t>0$ and $0<x\leq\underline h(t)$.

In view of Proposition \ref{p6}, it is obvious that (\ref{75}) and (\ref{75.0}) hold by the second inequations of (\ref{65}) and the comparison principle. The proofs are finished.
\end{proof}

Intuitively, in order to survive, the prey $u$ should spread faster than the predator $v$. Before the predator occupies a new habitat, the prey has colonized the habitat. Besides, both the prey and predator live in a region enclosed by free boundary caused only by the prey, which decides that the spreading speed of prey is no less than that of predator. Thus, it is rational to assume that $c_1\geq c_4$.
\begin{theorem}\label{th5.2}
Assume that $d({\mu}+c)\leq{\lambda}-b/m$. Then
\begin{equation}\label{68}
\lim\limits_{{\rho}\rightarrow\infty}\lim\limits_{t\rightarrow\infty}\frac{h(t)}{t}\geq c_5,
\end{equation}
and
\begin{equation}\label{69}
\liminf\limits_{t\rightarrow\infty}\inf\limits_{0<x<(c_5-\varepsilon)t}u(t,x)>0.
\end{equation}
\end{theorem}

\begin{proof}
The assumption $d({\mu}+c)\leq{\lambda}-b/m$ implies $c_4\leq c_1$. We can choose a small $\varepsilon>0$ such that $c_4+\varepsilon<c_1-\varepsilon$. Thus, by (\ref{80}), we have $ (c_4+\varepsilon)t<(c_1-\varepsilon)t<h(t)$ for all $t\geq\tau_1$ and ${\rho}\geq{\rho}_{\varepsilon}$. Denote $s:=\frac{2K}{{\lambda}-b/m}$, where $K$ is given by Theorem \ref{th1}. Recall that $v\leq K$ and
$\liminf\limits_{t\rightarrow\infty}\inf\limits_{0<x<(c_1-\varepsilon)t}u(t,x)\geq{\lambda}-{b}/{m}
$, and then there exists a large $\tau_2>0$ such that $v(t,x)\leq su(t,x)$ for all $t>\tau_2$ and $0<x<(c_1-\varepsilon)t$. Besides, it follows from (\ref{80}) that $v(t,x)=0$ when $x\geq(c_1-\varepsilon)t$. Let $\psi$ be defined by
$$
\begin{cases}
\psi_t-\psi_{xx}=({\lambda}-\frac{bs}{1+ms})\psi-\psi^2,~~&t>\tau_2,~0<x<g(t),\\
\psi_x(t,0)=\psi(t,g(t))=0,&t\geq\tau_2,\\
g^{\prime}(t)=-{\rho} \psi_x(t,g(t)),&t\geq\tau_2,\\
\psi(\tau_2,x)=u(\tau_2,x),~g(\tau_2)=h(\tau_2),&0\leq x\leq g(\tau_2).
\end{cases}
$$
By Proposition \ref{rm1}, we have $u(t,x)\geq\psi(t,x)$ and $h(t)\geq g(t)$ for $t>\tau_2$ and $0<x<g(t)$. According to Proposition \ref{p3}, we get $\lim\limits_{t\rightarrow\infty}\frac{g(t)}{t}=c({\rho},1,{\lambda}-\frac{bs}{1+ms},1)$
and $\lim\limits_{{\rho}\rightarrow\infty}c({\rho},1,{\lambda}-\frac{bs}{1+ms},1)=2\sqrt{{\lambda}-\frac{bs}{1+ms}}=:c_5$.
Thus,
$$\liminf\limits_{{\rho}\rightarrow\infty}\liminf\limits_{t\rightarrow\infty}\frac{h(t)}{t}\geq \lim\limits_{{\rho}\rightarrow\infty}\lim\limits_{t\rightarrow\infty}\frac{g(t)}{t}=c_5.
$$
So (\ref{68}) holds.

Similar to the proof of (\ref{74.0}), we can show that
$$
\liminf\limits_{t\rightarrow\infty}\inf\limits_{0<x<(c_5-\varepsilon)t}u(t,x)\geq\lambda-\frac{bs}{1+ms}>\lambda-\frac{b}{m}>0,
$$
which implies that (\ref{69}) holds.
\end{proof}

\begin{theorem}\label{th5.3}
Assume $0<m{\lambda}-b<b{\mu}/c$. For any given $\varepsilon\in(0,c_3)$, we have
$$
\liminf\limits_{t\rightarrow\infty}\inf\limits_{0<x<(c_3-\varepsilon)t}u(t,x)=\limsup\limits_{t\rightarrow\infty}\sup\limits_{0<x<(c_3-\varepsilon)t}u(t,x)=u^*,
$$
$$
\liminf\limits_{t\rightarrow\infty}\inf\limits_{0<x<(c_3-\varepsilon)t}v(t,x)=\limsup\limits_{t\rightarrow\infty}\sup\limits_{0<x<(c_3-\varepsilon)t}v(t,x)=v^*,
$$
where $(u^*,v^*)$ is defined by Theorem \ref{th3}.
\end{theorem}
\begin{proof}
Define
$$
C_{[0,K]}:=\{w\in C^{2}(D_{\infty}):0\leq w\leq K\},
$$
where $D_{\infty}$ and $K$ are given by Theorem \ref{th1}. Then $C_{[0,K]}$ is an invariant region of the solution $u(t,x)$, $v(t,x)$ for problem (\ref{Q}). Thus, we have
$$
\lim\limits_{t\rightarrow\infty}\inf\limits_{0<x<(c_3-\frac{\varepsilon}{2})t}(u(t,x),v(t,x))\geq\mathbf{0},~\lim\limits_{t\rightarrow\infty}\inf\limits_{0<x<(c_3-{\varepsilon})t}(u(t,x),v(t,x))\geq\mathbf{0}.
$$
by Theorem \ref{th6.1}. Then the conclusion can be proved by the same way as that of \cite[Lemma 4.6]{25}.
\end{proof}

\section{Discussion}
This paper is concerned with a ratio-dependent predator-prey system with a Neumann boundary condition on the left side indicating that the left boundary is fixed, and a free boundary $x=h(t)$ determined only by prey which describes the process of movement for prey species. We prove a spreading-vanishing dichotomy; in other words, the following alternative holds.
Either

(i)(the spreading case) the two species can spread successfully into $[0,\infty)$ (i.e., $h_{\infty}=\infty$) and the solution $(u,v)$ will stabilize at a positive equilibrium state; \\
or

(ii)(the vanishing case) the two species cannot spread to the whole space (i.e., $h_{\infty}\leq\Lambda$) and the prey will vanish eventually.\\
Meanwhile, we found a critical value
$$
\Lambda=\frac{\pi}{2}\sqrt{\frac{m}{m{\lambda}-b}},
$$
which can be called a ``spreading barrier" such that the prey will spread and successfully establish itself if it can break through this barrier ${\Lambda}$, or will vanish and never break through this barrier. In addition, we obtain the criteria for spreading and vanishing:

(i)Spreading happens if the size of initial habitat $h_0$ is more than or equal to $h^*$ (Theorem \ref{th4.2}(i)), or the moving parameter ${\rho}$ is large enough (${\rho}>{\rho}^0$) regardless of the initial habitat's size (Theorem \ref{th4.2}(ii)). 

(ii)Vanishing happens if the size of initial habitat $h_0$ is less than $h_*$ and the moving parameter ${\rho}$ is less than ${\rho}_0$ (Theorem \ref{th4.2}(iii)).

Finally, we get the asymptotic speed of $h(t)$ as $t\rightarrow\infty$. We want to draw a comparison with the uncoupled case.
If problem (\ref{Q}) is uncoupled (i.e., $b=c=0$), then the prey satisfies
$$
\begin{cases}
\phi_t-\phi_{xx}=\phi({\lambda}-\phi),~~~~&t>0,~0<x<\varsigma(t),\\
\phi_x(t,0)=\phi(t,\varsigma(t))=0,     &t\geq0,\\
\varsigma^{\prime}(t)=-{\rho}\phi_x(t,\varsigma(t)),&t\geq0,\\
\phi(0,x)=u_0(x), &0\leq x\leq\varsigma(0).
\end{cases}
$$
By Proposition \ref{p3}, we have
\begin{equation}\label{82}
\lim\limits_{{\rho}\rightarrow\infty}\lim\limits_{t\rightarrow\infty}\frac{\varsigma(t)}{t}=2\sqrt{{\lambda}}.
\end{equation}
And more notably, Theorem \ref{th6.1} and Theorem \ref{th5.2} show that when the spreading speed of prey is no lower than that of predator ($d(\mu+c)\leq\lambda-b/m$), then we have 
\begin{equation}\label{83}
2\sqrt{\lambda-\frac{bs}{1+ms}}\leq\lim\limits_{{\rho}\rightarrow\infty}\lim\limits_{t\rightarrow\infty}\frac{h(t)}{t}\leq 2\sqrt{\lambda}.
\end{equation}
The formulas (\ref{82}) and (\ref{83}) illustrate that the predator could decrease the prey's asymptotic propagation in the setting model studied in this paper. Furthermore, when we move to the right at a fixed speed less than $2\sqrt{d{\mu}}$, we will observe that the two species will stabilize at the unique positive equilibrium state; when we do this with a fixed speed in $(2\sqrt{d(\mu+c)}, 2\sqrt{\lambda-\frac{bs}{1+ms}})$, we can only see the prey; when we do this with a fixed speed more than $2\sqrt{\lambda}$, we could see neither for the two species are out of sight.

This paper has the same boundary conditions as paper \cite{wang9}, but it gets different conclusions from it. For the classical Lotka-Volterra type (\cite{wang9}), the sharp criteria with respect to the initial habitat $h_0$ can be obtained in some cases. While in this paper the predator does not impact the critical value $\Lambda$ and we cannot get the sharp criteria for spreading and vanishing with respect to the initial habitat $h_0$ in any case. Besides, we provide an estimate of the asymptotic speed of the free boundary and explain some realistic and significant spreading phenomena.

Some results are instructive in real life. On one hand, we confirm that alien species can have  serious impacts on the native species: two species can coexist (the spreading case), or more seriously, the invasive species (predator) may upset the ecological balance and wipe out the local species (prey) (the vanishing case). On the other hand, in order to control and eliminate the pest species (prey), we must take both approaches at the same time: (i)reduce the size of the initial habitat of the prey, (ii)decrease the coefficient of the free boundary. Significantly, it is useful to introduce its natural enemies (predator) when the initial habitat of the pest species is not very large.

\section*{\large{Acknowledgments}}
The author's work was supported by NSFC (12071316). The author would like to thank Professor Mingxin Wang who gave me invaluable lectures on free boundary. The author is also grateful for the anonymous reviewers for their helpful comments and suggestions.

\section*{\large{Declarations of interest: none.}}

\renewcommand{\theproposition}{A.\arabic{proposition}}
\setcounter{proposition}{0}
\setcounter{equation}{0}
\renewcommand{\theequation}{A.\arabic{equation}}
\section*{\textbf\large{A. Estimates of solutions to parabolic partial differential inequalities}}
Let $d$, $\beta$ be positive constants and $f(s)$ be a positive ${C}^1$ function for $s>0$.

\begin{proposition}\label{p1}
For any given $\varepsilon>0$ and $L>0$, there exist $l_{\varepsilon}>\max\{L,\frac{\pi}{2}\sqrt{d/\beta f(0)}\}$ and $T_{\varepsilon}>0$such that when the continuous function $w(t,x)\geq0$ satisfies
$$
\left\{
\begin{array}{ll}
w_t-dw_{xx}\geq(\leq) wf(w)(\beta-w),  &t>0,~0<x<l_{\varepsilon},\\
w_{x}(t,0)=0,             &t>0,\\
w(0,x)>0,                &0< x< l_{\varepsilon},
\end{array}
\right.
$$
and $w(t,l_{\varepsilon})\geq k$ if $k=0$, while $w(t,l_{\varepsilon})\leq k$ if $k>0$ for $t>0$, then we have
$$
w(t,x)>\beta-\varepsilon~(w(t,x)<\beta+\varepsilon), ~~~\forall~t>T_{\varepsilon},~x\in[0,L].
$$
Furthermore,
$$
\liminf\limits_{t\rightarrow\infty} w(t,x)>\beta-\varepsilon~(\limsup\limits_{t\rightarrow\infty} w(t,x)<\beta+\varepsilon)
$$
uniformly on $[0,L]$.
\end{proposition}

This proof is similar to the proof of \cite[Proposition 8.1]{wang6}, so we omit it.

\renewcommand{\theproposition}{B.\arabic{proposition}}
\setcounter{proposition}{0}
\setcounter{equation}{0}
\renewcommand{\theequation}{B.\arabic{equation}}
\section*{\textbf\large{B. Comparison Principle with Free Boundary}}
\begin{proposition}\label{lm3}
Suppose that $T>0$, $\overline u,~\overline v\in {C}^{1,2}(\mathcal O_T)$, $\overline h\in  {C}^1([0,T])$, where $\mathcal O_T=\{(t,x)\in \mathbb R^2:t>0, 0<x<\overline h(t)\}$, and $(\overline u, \overline v, \overline h)$ satisfies
$$
\left\{
\begin{array}{ll}
\overline u_t-\overline u_{xx}\geq(\alpha-\overline u)\overline u,~~ &t>0,0<x<\overline h(t),\\
\overline v_t-d\overline v_{xx}\geq(\beta-\overline v)\overline v, &t>0,0<x<\overline h(t),\\
\overline u_x(t,0)\leq 0,~\overline v_x(t,0)\leq 0,                  &t>0,\\
\overline u(t,\overline h(t))=\overline v(t,\overline h(t))=0,     &t>0,\\
\overline h^{\prime}(t)\geq -{\rho}\overline u_x(t,\overline h(t)),   &t>0.
\end{array}
\right.
$$
Let $\alpha={\lambda}$ and $\beta={\mu}+c$.
If
$$
\overline u(0,x)\geq u(0,x),~\overline v(0,x)\geq v(0,x),~\overline h(0)\geq h_0,~~x\in[0,\overline h(0)],
$$
then the solution $(u,v,h)$ of the problem {\upshape(\ref{Q})} satisfies
$$
u(t,x)\leq\overline u(t,x),~v(t,x)\leq\overline v(t,x),~~(t,x)\in (0,\infty)\times(0,h(t)),
$$
$$
h(t)\leq\overline h(t),~~~t\in(0,\infty).
$$
\end{proposition}

\begin{proposition}\label{rm1}
 $(\overline u,\overline v,\overline h)$ above is called an upper solution of the problem \eqref{Q}. We also can define upper solutions $(\overline u,\overline h)$ and $(\overline v,\overline h)$ like above. In like manner, we can defined lower solutions $(\underline u,\underline v,\underline h)$, $(\underline u,\underline h)$ and $(\underline v,\underline h)$ by letting $\alpha={\lambda}-b/m$, $\beta={\mu}$ and reversing all the inequalities above.
\end{proposition}

The proofs of Proposition \ref{lm3}, \ref{rm1}  are similar to \cite[Lemma 3]{wang6}. So we omit the proofs.

%%%%%%%%%%%%%%%%%%%%%%%%%%%%%%%%%%%[参考文献]%%%%%%%%%%%%%%%%%%%%%%%%%%%%%%%%%%%%%%%%%%%%%%%%%%%%%%
\section*{\large{References}}
\bibliographystyle{plain}\setlength{\bibsep}{0ex}
\scriptsize
\bibliography{mybibfile}
\end{document}